\documentclass{article}
\usepackage{latexsym}
\usepackage{amssymb,amsbsy,amsmath,amsfonts,amssymb,amscd,amsthm}
\usepackage{epsfig,graphicx,mathrsfs}
\usepackage{color}
\usepackage{subfigure}
\usepackage{pdfsync}
\usepackage{url}
\usepackage{algorithm}
\usepackage{algorithmic}

\newcommand{\ov}{\overline}
\newcommand{\R}{\mathbb{R}}

\newcommand{\E}{\mathbb{E}}

\newcommand{\beq}{\begin{equation}}
\newcommand{\eeq}{\end{equation}}
\def\a{\alpha}
\def\b{\beta}
\def\d{\delta}

\def\r{\rho}

\def\o{\omega}
\def\r{\rho}
\def\O{\Omega}
\def\G{\Gamma}

\def\t{\tau}
\def\e{\varepsilon}

\def\pd{\partial}

\newcommand{\cA}{{\mathcal A}}

\newcommand{\cE}{{\cal E}}

\newcommand{\fDz}{D_{(0,t]}^{1-\beta}}

\newcommand{\fdz}{\partial_{(0,t]}^{\beta}}

\newcommand{\fIz}{I_{(0,t]}^{\beta}}

\newcommand{\supp}{{\rm supp}}

\newcommand{\tgh}{{\tau_h}}
\newtheorem{theorem}{Theorem}[section]
\newtheorem{lemma}[theorem]{Lemma}
\newtheorem{definition}[theorem]{Definition}
\newtheorem{proposition}[theorem]{Proposition}

\newtheorem{remark}[theorem]{Remark}

\numberwithin{equation}{section}
\begin{document}
\title{\Large \bf A Hopf-Lax formula for   Hamilton-Jacobi equations   with Caputo~time derivative}
\author{Fabio Camilli\footnotemark[1] \and Raul De Maio\footnotemark[1]\and Elisa Iacomini\footnotemark[1]}
\date{\today}
\maketitle
\begin{abstract}
We prove a representation formula of   Hopf-Lax type for   the  solution of  a Hamilton-Jacobi equation
 involving Caputo time-fractional derivative. Equations of these type are associated with optimal control problems where the controlled
dynamics  is replaced by a time-changed stochastic process  describing the trajectory of a particle subject to random trapping effects.
\end{abstract}
 \begin{description}
    \item [{\bf AMS subject classification}:] 35R11,   26A33, 49L20.
     \item[{\bf Keywords}:]   Caputo time derivative, fractional Hamilton-Jacobi  equation, Hopf-Lax formula, subordinator.
\end{description}
\footnotetext[1]{Dip. di Scienze di Base e Applicate per l'Ingegneria,  ``Sapienza'' Universit{\`a}  di Roma, via Scarpa 16,
00161 Roma, Italy, ({\tt e-mail: fabio.camilli,raul.demaio,elisa.iacomini@sbai.uniroma1.it})}

\pagestyle{plain}
\pagenumbering{arabic}
\section{Introduction}
In the recent times, several classical parabolic equations have been revisited by replacing  the  standard derivative  with fractional ones
\cite{acv1,CDM,gn,ly,m,ty}. Fractional time derivatives are given   by convolution integral of the  time-derivative with   power-law kernels.
They arise  in several   phenomena in connection with anomalous diffusion   and   are typical for memory effects in complex systems  (see \cite{mk} for a review). The   probabilistic interpretation of the corresponding physical models leads to the study of
subdiffusive  or, more in general, non markovian  processes.  From a mathematical point of view, the presence of nonlocal terms with respect to the time variable poses several technical difficulties.\par
Aim of this paper is to study the connection  between Hamilton-Jacobi equations and anomalous diffusions,
recovering a subordinated version of the Hopf-Lax formula.
Consider   the Cauchy problem
\begin{equation}\label{HJ}
\left\{\begin{array}{ll}
  \pd_t u+H(Du)=0\quad & (x,t)\in Q, \\
 u(x,0)=g(x)& x\in\R^d,
\end{array}
\right.
\end{equation}
where the Hamiltonian $H$ is convex and superlinear and $Q=\R^d\times(0,\infty)$.
Then the classical Hopf-Lax formula
\begin{equation}\label{HL}
    u(x,t)=\min_{y\in\R^d}\left\{tL\left(\frac{x-y}{t}\right) + g(y) \right\},
\end{equation}
where  $L$ is the Legendre transform of  $H$,
gives  the unique viscosity solution of \eqref{HJ}. Moreover, if $g$ is Lipschitz continuous, then
$u$ is also Lipschitz continuous and it is the maximal almost everywhere (a.e.) subsolution of \eqref{HJ}.
 Formula \eqref{HL} is derived from the optimal control interpretation of  the  Cauchy problem. Indeed, the Hamilton-Jacobi equation in \eqref{HJ}
can be interpreted  as the dynamic programming equation satisfied by the value function of a control problem with dynamics
\begin{equation}\label{dyn}
    \left\{
\begin{array}{ll}
  \dot x(s)= a(s) &  \quad s\in (0,t), \\
  x(t)=x,
\end{array}
\right.
\end{equation}
where $a :(0,\infty)\to\R^d$  is the control variable, and cost functional
\begin{equation}\label{cost}
    J(x,t,a)=\int_0^t L(a(s))ds +g(x(0)).
\end{equation}
Since $L$ is independent of $(x,t)$,  straight lines are proved to be the minimizing trajectories
in \eqref{cost} and \eqref{HL} is so obtained (see \cite{E,L} for details).
In this paper we consider the Cauchy problem
\begin{equation}\label{HJf}
\left\{\begin{array}{ll}
\fdz u+H(Du)=0\quad & (x,t)\in Q, \\[8pt]
 u(x,0)=g(x) & x\in\R^d,
\end{array}
\right.
\end{equation}
where
$$
\fdz u(x,t)= \frac{1}{\G(1-\b)}\int_0^t\frac{\pd_\t u (x,\t)}{(t-\t)^\b}d\t
$$
is the Caputo time-fractional derivative of order $\b\in (0,1)$ of $u$.
To deduce a   Hopf-Lax formula for \eqref{HJf} we rely, as in the classical case,  on the optimal control interpretation of the problem.
Let $E_t$ be a continuous, nondecreasing stochastic process defined as the inverse of a $\beta$-stable subordinator $D_t$, i.e. $E_t := \inf\{\tau>0 : D_\tau>t\}$ for $t\geq 0$.   The stochastic process $X(t)=x(E(t))$, where $x(t)$ is given by \eqref{dyn}, solves the stochastic differential equation
\begin{equation}\label{dynf}
    \left\{
\begin{array}{ll}
  dX(t)= \bar a(t)dE_t, &  t\in (0,\infty) \\
  X(t)=x,
\end{array}
\right.
\end{equation}
where $\bar a(t)=a(E(t))$ for $a\in\cA$.   The subordinator  $E_t$ can be interpreted as a change of the time-scale  which introduces   trapping events  in the evolution of the process $X(t)$,     whereas, when  not trapped,     the particle moves according to the standard  dynamics $x(t)$.
 Define the cost functional
\begin{equation}\label{costf}
    J_\b(x,t, a)=\E_{x,t}\left\{  \int_0^tL(   a(s))dE_s +g(X(0))\right\}.
\end{equation}
It is clear that straight lines are still the optimal trajectories minimizing  \eqref{costf}, but traveled at a velocity which depends on the time scale $E_t$. We prove that the value function $u_\b$ associated to the time-changed control problem is given  by  the Hopf-Lax formula
 \[
u_\b(x,t)= \E_{x,t}\left[\min_{y\in\R^d}\left\{  E_t   L\left(\frac{x-y}{E_t}\right)+g(y)\right\}    \right].
 \]
The previous formula  is similar to \eqref{HL}, but it takes into account the  average  speed at which the straight trajectories are traveled.
We also prove that $u_\b$ is the maximal  subsolution  and an a.e. solution of   problem \eqref{HJf}, but we are
not able to prove that it is a viscosity solution in the sense of the definition introduced in \cite{gn,ty} (see Remark \ref{visco} for more details).
We can rewrite formula \eqref{HL} as the  convolution of the solution of \eqref{HJ} with a kernel given by the probability density function (PDF) of the process $E_t$, i.e.
\[u_\b(x,t)=\int_0^\infty u(x,s)\cE_\b(s,t)ds.\]
Employing a standard numerical solver for  \eqref{HJ} to compute $u$, we use   the previous formula  to illustrate with  some numerical examples the effect of the Caputo derivative on  control problems and fronts propagation.\par
The paper is organized as follows. In Section \ref{sec:hopf_lax}, we briefly recall some properties of the  subordinator process
 and we introduce the Hopf-Lax formula. Section \ref{sec:HJf} is devoted to the time-fractional
Hamilton-Jacobi equation. Finally, in Section \ref{sec:numerical}, some numerical examples are discussed in order  to stress the main differences with the classical theory.

\section{The  subordinator process and the Hopf-Lax formula}\label{sec:hopf_lax}
Let $\{D_\t\}_{\t\geq 0}$ be a  stable subordinator of order $\beta \in (0,1)$, i.e. a one-dimensional, non-decreasing L\'evy process whose PDF $g(s,\t)$ has Laplace transform equal to $e^{-\t s^{\beta}}$. The inverse stable  process $\{E_t\}_{t\geq 0}$, defined as the first passage time of the process $D_\t$ over the level $t$, i.e.
$$E_t = \inf\{\t>0 : D_\t > t\},$$
has sample paths which are continuous,  non-decreasing and such that $E_0 = 0$,  $E_t \to \infty$ as $t \to \infty$.
It is worthwhile  to observe that  $E_t$ does  not have stationary or independent increments. The process $E_t$  can be used to model systems with two time scales: a deterministic one given by the standard time $t$, referred to the external  observer, and  a stochastic one given by $E_t$, internal to the physical process (see \cite{mgz,ms,mst}).\\
We recall some basic properties of the process $E_t$ which we will exploit in the following
\begin{proposition}
For $t>0$, it holds:
\begin{itemize}
\item For any $\alpha>0$,   there exists a constant $C(\alpha, \beta) > 0$ such that
\begin{equation}\label{moment}
\E[E_t^\alpha] = C(\alpha, \beta) t^{\alpha\beta}.
\end{equation}
\item The process $E_t$  has PDF
\begin{equation}\label{pdf_trick}
\cE_\beta(s,t) = \frac{t}{\beta}s^{-1 - \frac{1}{\beta}}g(s,t).
\end{equation}
\end{itemize}
\end{proposition}

For the proof of the following    result we refer to \cite{mst}
\begin{proposition}
The function $\cE_\b(\cdot,t)$ is a weak solution of
\begin{equation}\label{eq_E}
\fdz \cE_\b(r,t)=-\pd_r \cE_\b(r,t), \qquad r\in (0,\infty).
\end{equation}
\end{proposition}
In the following, we assume that
\begin{align}\label{hyp_H}
    \text{$H:\R^d\to \R$ is convex and }\lim_{|p|\to \infty} \frac{H(p)}{p}=+\infty;\\
\label{hyp_g}
    \text{$g:\R^d\to \R$ is Lipschitz continuous, bounded.}
\end{align}
 The Legendre transform $L$ of $H$, defined by $L(q)=\sup\{pq-H(p)\}$,   is well defined,
convex and superlinear. Let  $u_\b$ be the value function  of the stochastic control problem
\eqref{dynf}-\eqref{costf}, i.e.
\begin{equation}\label{valuef}
    u_\b(x,t)=\inf_{\a\in\cA}J_\b(x,t,\a),
\end{equation}
where $\cA:=\{\a:(0,\infty)\to\R^d:\,\a\text{ is a progressively measurable process}\}$.
\begin{proposition}
 The value function $u_\b$ defined in \eqref{valuef} is given by the Hopf-Lax formula
 \begin{equation}\label{HLf}
     u_\b(x,t)= \E_{x,t}\left[\min_{y\in\R^d}\left\{  E_t   L\left(\frac{x-y}{E_t}\right)+g(y)\right\}    \right].
 \end{equation}
\end{proposition}
\begin{proof}
Fix $(x,t)\in\R^d\times (0,+\infty)$. Let $Y:\O\to\R^d$ be a r.v. such that
\begin{equation}\label{id}
\min_{y\in\R^d}\left\{  E_t   L\left(\frac{x-y}{E_t}\right)+g(y)\right\}=   E_t   L\left(\frac{x-Y}{E_t}\right)+g(Y).
\end{equation}
Note that $Y$ is well defined since, by \eqref{hyp_H} and \eqref{hyp_g}, the minimum in the LHS of
\eqref{id} is achieved for any $\o\in\Omega$.
For the control law $\bar a(s)=(x-Y)/E_t$, consider the solution $X(s)$ of \eqref{dynf}.  Then
\[X(s)=Y+\frac{x-Y}{E_t}E_s\]
and therefore  $X(0)=Y$ (recall that $E_0=0$) and $X(t)=x$. Hence
\begin{align*}
    u_\b(x,t)&\le \E_{x,t}\left\{ \int_0^tL\left( \frac{x-Y}{E_t}\right)dE_s +g(X(0))\right\}=\E_{x,t}\left\{  E_t   L\left(\frac{x-Y}{E_t}\right)+g(Y)\right\}\\
    &=\E_{x,t}\left[ \min_{y\in\R^d}\left\{  E_t   L\left(\frac{x-y}{E_t}\right)+g(y)\right\}\right].
\end{align*}
We prove the reverse inequality. Given a control $\a\in \cA$,  by the convexity of $L$ we have
\[
L\left(\frac{1}{E_t}\int_0^{t} \a(s)dE_s\right)\le \frac{1}{E_t}\int_0^{t} L( \a(s))dE_s.
\]
If $X(t)$ is the solution of \eqref{dynf}, since  $\int_0^{t} \a(s)dE_s=x-X(0)$, by the previous inequality we get
 \begin{align*}
\E_{x,t}\left\{ \int_0^{t} L( \a(s))dE_s +g(X(0))\right\} &\ge \E_{x,t}\left\{E_tL\left(\frac{x-X(0)}{E_t}\right)+g(X(0))\right\}\\
&\ge \E_{x,t}\left[\min_{y\in\R^d}\left\{  E_t   L\left(\frac{x-y}{E_t}\right)+g(y)\right\}    \right]
\end{align*}
and, for the arbitrariness of $a\in\cA$,
\[u_\b(x,t)\ge  \E_{x,t}\left[\min_{y\in\R^d}\left\{  E_t   L\left(\frac{x-y}{E_t}\right)+g(y)\right\}    \right]. \]
\end{proof}
In order to prove some regularity properties of the   function $u_\b$, we need  a preliminary result.
\begin{lemma}\label{ppd}
For $s\in [0,t)$, we have
\begin{equation}\label{L1}
   u_\b(x,t)=\E_{x,t}\left[ \min_{y\in\R^d}\left\{( E_t -E_s)  L\left(\frac{x-y}{E_t-E_s}\right)+u_\b(y,s)\right\}\right].
\end{equation}
\end{lemma}
\begin{proof}
For $y\in \R^d$, let $Z:\O\to \R^d$ be a r.v. such that
\[u_\b(y,s)=\E_{x,t}\left[E_sL\left(\frac{y-Z}{E_s}\right)+g(Z)\right]\]
By the identity
\[\frac{x-Z}{E_t}=\left(1-\frac{E_s}{E_t}\right)\frac{x-y}{E_t-E_s}+ \frac{E_s}{E_t}\frac{y-Z}{E_s}\]
and  by the convexity of $L$, we get
\[
L\left(\frac{x-Z}{E_t}\right)\le \left(1-\frac{E_s}{E_t}\right)L\left(\frac{x-y}{E_t-E_s}\right)+ \frac{E_s}{E_t}L\left(\frac{y-Z}{E_s}\right).
\]
Therefore
\begin{align*}
     u_\b(x,t)&=  \E_{x,t}\left[\min_{y\in\R^d}\left\{  E_t   L\left(\frac{x-y}{E_t}\right)+g(y)\right\}\right]\le \E_{x,t}\left[ E_t L\left(\frac{x-Z}{E_t}\right)+g(Z)\right]\\
            &\le  \E_{x,t}\left[ (E_t-E_s)L\left(\frac{x-y}{E_t-E_s}\right)+E_sL\left(\frac{y-Z}{E_s}\right) +g(Z)\right]\\
            &=\E_{x,t}\left[ (E_t-E_s)L\left(\frac{x-y}{E_t-E_s}\right)+u_\b(y,s)\right].
     \end{align*}
Since the previous inequality holds for any $y\in\R^d$, we get
\[
 u_\b(x,t)\le \E_{x,t}\left[ \min_{y\in\R^d}\left\{ (E_t-E_s)L\left(\frac{x-y}{E_t-E_s}\right)+u_\b(y,s)\right\}\right].
\]
To prove the reverse inequality, let $W$ be a r.v. such that
\[
u_\b(x,t)=\E_{x,t}\left[ \min_{y\in\R^d}\left\{  E_t   L\left(\frac{x-y}{E_t}\right)+g(y)\right\}\right]=\E_{x,t}\left[ E_t   L\left(\frac{x-W}{E_t}\right)+g(W)\right].\]
If $Y:\Omega\to\R^d$ is a r.v., since
\[\E_{x,t}\left[u_\b(Y ,s)\right]\le\E_{x,t}\left[ E _s  L\left(\frac{Y -W}{E_s}\right)+g(W)\right], \]
it follows that
\begin{equation}\label{L4}
 u_\b(x,t)\ge \E_{x,t}\left[ E _t   L\left(\frac{x-W}{E_t}\right)-E _s  L\left(\frac{Y-W}{E_s}\right)+u_\b(Y,s) \right].
\end{equation}
Set $Y=\frac{E_s}{E_t}x+ (1-\frac{E_s}{E_t})W$. Then $\frac{x-Y}{E_t-E_s}=\frac{x-W}{E_t}=\frac{Y-W}{E_s}$ and by
\eqref{L4}
\begin{align*}
    u_\b(x,t)&\ge \E_{x,t}\left[ ( E_t -E_s)  L\left(\frac{x-Y}{E_t-E_s}\right)+u_\b(Y,s)\right]\\
    &\ge \E_{x,t}\left[  \min_{y\in\R^d}\left\{( E_t -E_s)  L\left(\frac{x-y}{E_t-E_s}\right)+u_\b(y,s)\right\}\right].
\end{align*}
\end{proof}
\begin{remark}
Arguing as in Lemma \ref{ppd}, it is also possible to prove that if $\t:\O\to [0,t)$ is a stopping time, then
\begin{equation}\label{L1bis}
   u_\b(x,t)=\E_{x,t}\left[ \min_{y\in\R^d}\left\{( E_t -E_\t)  L\left(\frac{x-y}{E_t-E_\t}\right)+u_\b(y,\t)\right\}\right].
\end{equation}
\end{remark}
\begin{proposition}\label{regularity}
We have
\begin{enumerate}
 \item[(i)] For all $x$, $\ov x\in\R^d$,  $t\in (0,\infty)$
  \begin{equation}\label{lip_x}
    |u_\b(x,t)-u_\b(\ov x,t)|\le L_g |x-\ov x|,
  \end{equation}
  where $L_g$ is the Lipschitz constant of the initial datum $g$.
\item[(ii)] There exists a constant $C$ such  that for all   $x \in\R^d$,  $t\in (0,\infty)$
  \begin{equation}\label{est_g}
    |u_\b(x,t)-g(x)|\le C t^\b.
  \end{equation}
\item[(iii)] There exists a constant $C$ such  that for all   $x \in\R^d$,  $t$, $\ov t\in (0,\infty)$, $\ov t< t$,
\begin{equation}\label{lip_t}
     |u_\b(x,t)-u_\b(x,\ov t)|\le C (t-\ov t)^\b.
\end{equation}
\end{enumerate}
\end{proposition}
\begin{proof}
Fixed $t>0$,  $x$, $\ov x\in\R^d$,  let $Z:\O\to \R^d$ be a r.v. such that
\[E_tL\left(\frac{x-Z}{E_t}\right)+g(Z)=\min_{y\in\R^d}\left\{E_tL\left(\frac{x-y}{E_t}\right)+g(y)\right\}.\]
Then
\begin{align*}
    u_\b(\ov x,t)-u_\b(x,t)&=\E_{x,t}\left[  \min_{y\in\R^d}\left\{E_t  L\left(\frac{\ov x-y}{E_t}\right)+g(y)\right\}\right]-\E_{x,t}\left[E_tL\left(\frac{x-Z}{E_t}\right)+g(Z)\right]  \\
    &\le \E_{x,t}\left[ E_tL\left(\frac{x-Z}{E_t}\right)+g(\ov x-x+Z) \right]-\E_{x,t}\left[E_tL\left(\frac{x-Z}{E_t}\right)+g(Z)\right]\\
    &\le  \E_{x,t}\left[g(\ov x-x+Z)-g(Z)\right]\le L_g|\ov x-x|.
\end{align*}
Exchanging the role of $x$, $\ov x$, we get \eqref{lip_x}.\\
Fix $x\in\R^d$ and $t>0$.
 We recall that    for any $t, \gamma > 0$, the  $\gamma$-moment of $E_t$  is given by
Then, setting $y=x$ in the RHS of \eqref{HLf} and recalling \eqref{moment}, we get
\begin{equation}\label{g1}
    u_\b(x,t) \le \E_{x,t}\left[  E_t  L(0)+g(x)\right]=L(0)\E_{x,t}[E_t]+ g(x)=L(0)c(\b,1)t^\b+g(x).
\end{equation}
To get the other inequality in \eqref{est_g}, we observe that
\begin{align}
    u_\b(x,t)&-g(x)=\E_{x,t}\left[  \min_{y\in\R^d}\left\{E_t  L\left(\frac{x-y}{E_t}\right)+g(y)-g(x)\right\}\right]\nonumber\\
                         &\ge \E_{x,t}\left[  \min_{y\in\R^d}\left\{E_t  L\left(\frac{x-y}{E_t}\right)-L_g|x-y|\right\}\right]
                         =- \E_{x,t}\left[E_t \max_{z\in\R^d}\left\{-L(z)+L_g|z|\right\}\right]\label{L7}\\
                         &=\E_{x,t}\left[ -E_t\max_{|w|\le L_g}\max_{z\in\R^d}\left\{|w|-L(w)\right\}\right]
                         =-\max_{|w|\le L_g}\{H(w)\}\E_{x,t}[E_t].\nonumber
\end{align}
By \eqref{moment} and \eqref{g1}, we get \eqref{est_g}.\\
To prove \eqref{lip_t}, fix $x\in\R^d$ and $0<\ov t<t$. Setting $y=x$, $s=\ov t$ in the RHS of \eqref{L1} and recalling \eqref{moment}, we get
\begin{align*}
    u_\b(x,t) &\le \E_{x,t}\left[  (E_t -E_{\ov t})L(0)+u_\b(x,\ov t)\right]=L(0)\E_{x,t}[E_t-E_{\ov t}] +u_\b(x,\ov t) \\
    &=L(0)c(\b,1)(t^\b-{\ov t}^\b)+u_\b(x,\ov t) \le L(0)c(\b,1)(t-\ov t)^\b+u_\b(x,\ov t).
\end{align*}
On the other side, by \eqref{L1} with $s={\ov t}$ and \eqref{lip_x}, arguing as in \eqref{L7}, we have
\begin{align*}
    u_\b(x,t) -u_\b(x,\ov t) &=\E_{x,t}\left[ \min_{y\in\R^d}\left\{( E_t -E_{\ov t})  L\left(\frac{x-y}{E_t-E_{\ov t}}\right)+u_\b(y,{\ov t})
     -u_\b(x,\ov t)\right\}\right]\\
                  &\ge \E_{x,t}\left[ \min_{y\in\R^d}\left\{( E_t -E_{\ov t})  L\left(\frac{x-y}{E_t-E_{\ov t}}\right)-L_g|x-y|\right\}\right]\\
                  &=\E_{x,t}\left[ -(E_t-E_{\ov t})\max_{|w|\le L_g}\max_{z\in\R^d}\left\{|w|-L(w)\right\}\right]\\
                  &=-\max_{|w|\le L_g}\{H(w)\}\E_{x,t}[E_t-E_{\ov t}]\ge -\max_{|w|\le L_g}\{H(w)\}c(\b,1)(t-\ov t)^\b.
\end{align*}
\end{proof}
\section{The fractional Hamilton-Jacobi equation}\label{sec:HJf}
We exploit the results of the previous section to show that the function $u_\b$ given by \eqref{HLf} is a solution of the Cauchy problem \eqref{HJf}.
\begin{definition}
A function $v\in C^0(\overline Q)$ is said to be    a.e. subsolution  of \eqref{HJf}  is $Dv\in L^\infty (Q)$, $\pd_t v(\cdot,x)\in L^1_{\textrm{loc}} (0,\infty)$  for all $x\in\R^d$ and
\begin{eqnarray}
     &  \fdz v+H(Dv)\le 0\quad  &\text{a.e. in $Q$},\label{subHJf}\\
  &   v(x,0)\le g(x)&x\in\R^d.\label{subIC}
    \end{eqnarray}
\end{definition}
Observe that, since    $\pd_t v(\cdot,x)\in L^1_{\textrm{loc}} (0,\infty)$  for all $x\in\R^d$, the fractional derivative $\fdz v(x,t)$ is well defined for any $(x,t)\in Q$.
\begin{proposition}\label{sol}
 The    function $u_\b$    is the maximal a.e. subsolution of \eqref{HJf}. In addition,  $u_\b$ is   a.e. solution of
 \eqref{HJf} in $Q$.
\end{proposition}
\begin{proof}
By Proposition \ref{regularity}, we have that  $Du_\b\in L^\infty(Q)$, $\pd_tu_\b(\cdot,x)\in L^1_{\textrm{loc}}
(0,\infty)$  for all $x\in\R^d$. Given $(x,t)\in\R^d\times(0,\infty)$ such that $Du(x,t)$ exists, we  prove that for any $q\in\R^d$,
\begin{equation}\label{claim1}
    \fdz u_\b(x,t)  + Du_\b(x,t)\cdot q -L(q)\le 0.
\end{equation}
Indeed, fix $q\in\R^d$ and $h>0$. Consider  the  control law $\a(s)\equiv q$. Then the solution $X(s)$ of \eqref{dynf}  is given by $X(s)=x-(E_t-E_s)q$.
Define the stopping time
\[\tgh=\sup\{ s\in (t-h,t):\, |X(s)-x|=h\}.\]
By \eqref{L1bis},  for $\t=\t_h$ and $y=X(\tgh)$, we have
\[
u_\b(x,t)\le \E_{x,t}\left[ (E_t-E_\tgh)L\left(\frac{x-X(\tgh)}{ E_t-E_\tgh } \right)+u_\b(X(\tgh),\tgh)  \right].
\]
Since $q=(x-X(\tgh))/(E_t-E_\tgh)$,
\begin{equation}\label{sol1}
 \E_{x,t}\left[u_\b(X(t),t)-u_\b(X(\tgh),\tgh)\right]\le L(q)\ \E_{x,t}\left[E_t-E_\tgh\right]
\end{equation}
By Ito's formula \cite{k}, we also have
\begin{align}
    & \E_{x,t}\left[u_\b(X(t),t)-u_\b(X(\tgh),\tgh)\right]= \E_{x,t}\left[\int_\tgh^t du_\b(X(s),s)\right]\nonumber\\
    & \E_{x,t}\left[\int_\tgh^t \pd_s u_\b(X(s),s)ds+\int_\tgh^tDu_\b(X(s),s)dX(s)\right]\label{sol1a}\\
    & =   \E_{x,t}\left[\int_\tgh^t \pd_s u_\b(X(s),s)ds+\int_\tgh^tDu_\b(X(s),s)\cdot q\,dE_s \right].\nonumber
\end{align}
Therefore, recalling \eqref{sol1},
\begin{align*}
     \E_{x,t}&\left[\int_\tgh^t \pd_s u_\b(X(s),s)ds+\int_\tgh^t\left(Du_\b(X(s),s)\cdot q-L(q)\right)dE_s \right]
     =\E_{x,t}\left[\int_\tgh^t \pd_s u_\b(X(s),s)ds\right. \\
    & \left. +\int_{0}^{+\infty} \left(\int_0^r \left(Du_\b(Y(s),D_s)\cdot q-L(q)\right)ds\right)(\cE_\b(r,t)-\cE_\b(r,\tgh) ) dr\right]    \le 0\nonumber
\end{align*}
where $D_s$ is the inverse of $E_t$, i.e. $E_{D_s}=s$.
Dividing the previous inequality by $h$ and passing to the limit for $h\to 0^+$, by the Dominated Convergence Theorem we get
\begin{equation}\label{sol2}
    \pd_t u_\b(x,t)+  \E_{x,t}\left[\int_{0}^{+\infty} \left(\int_0^r \left(Du_\b(Y(s),D_s)\cdot q-L(q)\right)ds\right)\pd_t \cE_\b(r,t) dr\right]    \le 0.
\end{equation}
Set $\Phi(r)=\int_0^r (Du_\b(Y(s),D_s)\cdot q-L(q))ds$. Since   $\cE_\b$ is a solution
of \eqref{eq_E}, then  we have (see \cite[Lemma 4.2]{CDM})
\[
\pd_t \cE_\b(r,t)=-\fDz [\pd_r \cE_\b(r,t)] - \delta_0(r)\delta_0(t), \qquad r\in (0,\infty).
\]
where $\fDz$ is      the Riemann-Liouville derivative   of order $1-\b$, which is defined for  a continuous function  $f:[0,t]\to \R$
 by
\[
   \fDz  f(t):=\frac{1}{\G(\b)}\frac{d}{dt} \int_0^t f (\t)\frac{1}{(t-\t)^{1-\b}}d\t.
\]
Therefore
\begin{align}
 \E_{x,t}&\left[\int_{0}^{+\infty}\Phi(r) \pd_t \cE_\b(r,t) dr\right]=-\E_{x,t}\left[\int_{0}^{+\infty}\Phi(r) \fDz \pd_r \cE_\b(r,t)  dr\right]\nonumber\\
             &- \Phi(0)\delta_0(t) =-\E_{x,t}\left[ \fDz\left(\int_{0}^{+\infty}\Phi(r) \pd_r \cE_\b(r,t)dr\right)  \right]\label{sol13}\\
             &=-\E_{x,t}\left[\fDz\left(\left[ \Phi(r)\cE_\b(r,t)\right]_0^{+\infty}-\int_{0}^{+\infty}  \pd_r\Phi(r)  \cE_\b(r,t)dr\right)  \right]\nonumber
 \end{align}
 Since $\lim_{r\to+\infty}\cE_\b(r,t)=0$ and $\pd_r\Phi(r)=Du_\b(Y(r),D_r)\cdot q-L(q)$, we have
 \begin{equation}\label{sol3}
 \begin{split}
  \E_{x,t}&\left[\int_{0}^{+\infty}\Phi(r) \pd_t \cE_\b(r,t) dr\right]= \E_{x,t}\left[\fDz\left(\int_{0}^{+\infty}\cE_\b(r,t)(Du_\b(Y(r),D_r)\cdot q-L(q))dr\right)\right]\\
  =\E_{x,t}&\left[\fDz\left(Du_\b(X(t),t)\cdot q-L(q)\right)\right]=\fDz\left(Du_\b(x,t)\cdot q-L(q)\right).
  \end{split}
 \end{equation}
 Replacing \eqref{sol3} in \eqref{sol2}, we get
 \begin{equation}\label{HJrl}
\pd_t u_\b(x,t)+\fDz\left(Du_\b(x,t)\cdot q-L(q)\right)\le 0.
 \end{equation}
 Applying the fractional integral $\fIz\cdot= \frac{1}{\G(\b)}\int_0^t \frac{\cdot}{(t-\t)^{1-\b}}d\t$ to the previous equation, we finally get
\begin{equation}\label{HJc}
 \fdz u_\b (x,t)+Du_\b(x,t)\cdot q-L(q)\le 0,
\end{equation}
 hence the claim \eqref{claim1}. It follows that   $u_\b$ is an a.e. subsolution of \eqref{HJf}. \par
 We now prove that $u_\b$ is the maximal a.e. subsolution.
 Assume by contradiction  that there exist   $(x,t)\in Q$, $\e>0$ and  a a.e. subsolution $v$  of \eqref{HJf}   such that
 \begin{equation}\label{contrad}
    u_\b(x,t)\le v(x,t)-2\e.
 \end{equation}
It is not restrictive to assume that that $v\in C^1(Q)$ (see   Lemma \ref{reg_sub} at the end of the proof).
Let $\a$ be  a $\e$-optimal control for $u_\b$, i.e.
\[u_\b(x,t)\ge\E_{x,t}\left\{\int_0^tL(\a(s))dE_s +g(X(0))\right\}-\e,\]
where $X(s)$ is given by the solution of \eqref{dynf} corresponding to $\a$. By Ito's formula
\begin{align}
v&(x,t)=     \E_{x,t}\left[v(X(t),t)\right]= \E_{x,t}\left[v(X(0),0)+\int_0^t dv(X(s),s)\right]\nonumber\\
    & \E_{x,t}\left[v(X(0),0)+\int_0^t \pd_s v(X(s),s)ds+\int_0^tDv(X(s),s)\cdot \a(s)dE_s\right]\nonumber\\
    & \le   \E_{x,t}\left[g(X(0))+\int_0^t \pd_s v(X(s),s)ds+\int_0^t\big(L(\a(s))+H(Dv(X(s),s))\big)\,dE_s \right]\label{sol1a}\nonumber\\
    &\le   u_\b(x,t)+\e+\E_{x,t}\left[\int_0^t \pd_s v(X(s),s)ds+\int_0^t H(Dv(X(s),s)) \,dE_s\right]\nonumber\\
    &=u_\b(x,t)+\e+\E_{x,t}\left[\int_0^t \pd_s v(X(s),s)ds+\int_0^\infty\left(\int_0^r H(Dv(Y(s),D_s))ds\right)(\cE_\b(r,t)-\cE_\b(r,0))dr\right]\nonumber\\
   & =u_\b(x,t)+\e+\E_{x,t}\left[\int_0^t \pd_s v(X(s),s)ds+\int_0^\infty\Phi(r)(\cE_\b(r,t)-\cE_\b(r,0))dr\right]\nonumber
\end{align}
 where $\Phi(r)=\int_0^r H(Dv(Y(s),D_s))ds$. Integrating \eqref{eq_E} and observing that $\cE_\b(r,0)=0$ for all $r\in [0,+\infty)$, we have
\[
\cE_\b(r,t)=-\fIz(\pd_r\cE_\b(r,\cdot))+\delta_0(r)H(r),
\]
where $\d_0$ and $H$ are the Dirac function at $0$ and the Heaviside function. Performing a computation similar to \eqref{sol13}, we have that
 \begin{align*}
&\E_{x,t}\left[\int_0^\infty\Phi(r)\cE_\b(r,t))dr\right] =-\E_{x,t}\left[\fIz\big(\int_0^\infty\Phi(r)\pd_r\cE_\b(r,\cdot)\big)dr\right]\\
&=\E_{x,t}\left[\fIz\big(\int_0^\infty\pd_r\Phi(r)\cE_\b(r,\cdot)\big)dr\right]=\E_{x,t}\left[\fIz\left(\int_0^\infty H(Dv(Y(r),D_r))\cE_\b(r,\cdot)dr\right)\right]\\
&=\E_{x,t}\left[   \fIz[H(D v(X(\cdot),\cdot))]\right].
\end{align*}
Replacing the previous identity in \eqref{sol1a}, we finally get that
 \begin{equation}\label{sol15}
 v(x,t)\le u_\b(x,t)+\e+\E_{x,t}\left[\int_0^t \pd_s v(X(s),s)ds+\fIz[H(D v(X(\cdot),\cdot))] \right].
 \end{equation}
Since   $v$ is a  $C^1$ subsolution of \eqref{HJf}, by applying the operator $\fIz\cdot$ to the equation satisfied  by $v$ we get
\[
\int_0^t\pd_s v(x,s)ds+ \fIz \left[H(D v(x,\cdot))\right]\le 0\qquad \forall (x,t)\in Q.
\]
Replacing the previous inequality in \eqref{sol15}, we get a contradiction to \eqref{contrad}.\par
We finally prove that $u_\b$ satisfies \eqref{HJf} a.e. in $Q$. Assume by contradiction that there exists $(x_0,t_0)\in Q$ and   $\d$, $\e$ positive such that, defined $U=(x_0-\d,x_0+\d)\times(t_0-\e,t_0+\e)$, we have
\begin{equation}\label{contrad1}
    \fdz u_\b+H(Du_\b)\le -2\d<0\qquad \text{a.e. in $U$.}
\end{equation}
Define the function $\phi(t)=(t-(t_0-\e))\chi_{(t_0-\e,t_0]}+((t_0+\e)-t)\chi_{(t_0,t_0+\e)}$ for $t\in\R$, where $\chi_{[a,b]}$ is the characteristic function of the interval $[a,b]$. Set $C_\b=\e^{\b}/\G(2-\b)$ and observe that $\fdz \phi(t)=(t-t_0+\e)^\b/\G(2-\b)$ for $t\in (t_0-\e,t_0]$, $\fdz \phi(t)=C_\b-(t-t_0)^\b/\G(2-\b)$ for
$t\in (t_0,t_0+\e)$ and $\fdz \phi(t)=0$ otherwise. Hence $\fdz \phi(t)\le C_\b$ for all $t\in\R$. Defined $\bar u(x,t)=u_\b(x,t)+\frac{\d}{C_\b}\phi(t)$
for $(x,t)\in Q$, by \eqref{contrad1} we have
\[
\fdz \bar u +H(D\bar u)\le \fdz u_\b+H(Du_\b)+ \d \le -\d\qquad \text{a.e. in $U$.}
\]
Moreover, if $\e$ is small enough in such a way that $t_0-\e>0$, it follows that  $\bar u(x,0)=u_\b(x,0)=g(x)$ and therefore $\bar u$ is an a.e. subsolution of \eqref{HJf}.
Since $\bar u(x_0,t_0)=u_\b(x_0,t_0)+\frac{\d}{C_\b}$,  we get  a contradiction to the maximality of $u_\b$ among the subsolutions of \eqref{HJf}.
\end{proof}

\begin{lemma}\label{reg_sub}
Let $v$ be a subsolution of \eqref{HJf}. Then there exists a sequence of subsolutions $v_\d\in C^1(Q)$ such that $v_\d$
tends to $v$ locally uniformly for $\d\to 0$.
\end{lemma}
\begin{proof}
Given a subsolution $v$, we define $v(x,t)=v(x,0)$ per $t\in(-\infty,0)$, hence we can write
\[\fdz v(x,t)=\frac{1}{\G(1-\b)}\int_0^t\frac{\pd_\t v(x,t-\t)}{\t ^\b}d\t=\frac{1}{\G(1-\b)}\int_0^{\infty}\frac{\pd_\t v(x,t-\t)}{\t ^\b}d\t.\]
Let $v_\d=v*\r_\d(x,t)$ where $\r_\d$ is a standard mollifier in $\R^{d+1}$, i.e. $\r_\d(\cdot)=\frac{1}{\d^{d+1}}\r\left(\frac{\cdot}{\d}\right)$
with $\r$ a smooth function such that $\supp\, \rho\,\subset\{|x|<1,|t|<1\}$ and $\int_{\R^{d+1}}\r dxdt=1$. Then $v_\d\to v$ locally uniformly for $\d\to 0$
and by convexity
\begin{equation}\label{conv}
    H(Dv_\d(x,t))\le  (H(Dv)*\r_\d)(x,t)\qquad\forall (x,t)\in Q.
\end{equation}
Moreover
\begin{align*}
    (\fdz v(x,t)*\r_\d)(x,t)&=\int_{\R^{d+1}} \frac{1}{\G(1-\b)} \left(\int_0^{\infty} \frac{\pd_t v(y,s-\t)}{\t ^\b}d\t\right)\r_\d(x-y,t-s)dsdy\\
  &=  \frac{1}{\G(1-\b)}\int_0^{\infty}\left( \int_{\R^{d+1}}  \pd_t v(y,s-\t) \r_\d(x-y,t-s) dsdy \right)\frac1{\t ^\b}d\t\\
 & =  \frac{1}{\G(1-\b)}\int_0^{\infty}\left( \int_{\R^{d+1}}  \pd_t v(y,r) \r_\d(x-y,t-\t-r) drdy \right)\frac1{\t ^\b}d\t\\
& =\frac{1}{\G(1-\b)}\int_0^{\infty}( \pd_tv*\r_\d) (x,t-\t)\frac1{\t ^\b}d\t=\fdz v_\d(x,t).
\end{align*}
Replacing the previous identity and \eqref{conv} in \eqref{subHJf} , we get
\[\fdz v_\d(x,t)+  H(Dv_\d(x,t))\le 0\qquad \forall (x,t)\in Q.\]
Since $\|v-v_\d\|_\infty\le C\d$, with $C$ depending on $\|Dv\|_\infty$, by subtracting
$C\d$ to $v_\d$, we have that $v_\d$ also satisfies $\eqref{subIC}$.
\end{proof}
\begin{remark}\label{visco}
It is well known that Hamilton-Jacobi equations such as \eqref{HJ} in general do not admit classical solutions and the
correct notion of weak solution is the one of  viscosity solution (\cite{bcd}).
A theory of viscosity solutions for a general class of  Hamilton-Jacobi equations with Caputo time derivative have been recently developed  in  \cite{gn,ty}. However, in these papers, the connection between   Hamilton-Jacobi equations and the corresponding optimal control theory
has not been pursued. In Theorem \ref{sol}, we establish this connection for a.e. (sub-)solutions, but we are not able to show the
corresponding property for viscosity solutions. Indeed, in the proof of the subsolution and supersolution conditions, applying the Ito's formula  as in the classical viscosity solution argument, we get an equation involving Riemann-Liouville time derivative, see for example \eqref{HJrl}.
The delicate point is that, for passing from \eqref{HJrl} to \eqref{HJc}, we   perform a fractional integration  and therefore we need that the
equation is satisfied globally, while the notion of viscosity solution is only local.
\end{remark}
\section{Integral formula and numerical examples }\label{sec:numerical}
We   propose some  examples in order to show   a comparison between  $u$, the solution of the classical Hamilton-Jacobi  equation \eqref{HJ},  and  $u_\b$,
the solution  of the time-fractional Hamilton-Jacobi equation \eqref{HJf}.
We start  rewriting formula \eqref{HLf} as
\begin{equation}\label{HLfbis}
u_\b(x,t)= \int_0^{\infty} \min_{y\in\R^d} \left\{rL\left(\frac{x-y}{r}\right)+g(y)\right\}\cE_\b(r,t)dr=\int_0^\infty u(x,r)\cE_\b(r,t)dr
\end{equation}
where $u$ is given by the formula \eqref{HL} and $\cE_\b(\cdot,t)$ is the PDF of
$E_t$. Recalling \eqref{pdf_trick},   \eqref{HLfbis} can be also rewritten as
\begin{equation}\label{HJfnum}
u_{\beta}(x,t)=\frac{t}{\beta} \int_0^\infty u(x,s)\frac{g_\beta(ts^{-1/\beta})}{s^{1 + \frac{1}{\beta}}}ds.
\end{equation}
We will use formula \eqref{HJfnum} to compute the function $u_\b$.  We assume  that the function $u$ is known in order to avoid additional numerical errors
due to its approximation which could further affect $u_{\beta}$ and hide some important properties. Moreover we approximate the integral by a quadrature formula and we
 employ the Matlab toolbox   Stable Distribution \cite{mat} to compute $g_{\beta}(s)$.   The        toolbox
requires    4 parameters $(\alpha, \hat\beta, \gamma, \delta)$ in order to compute a stable distribution (see \cite{pen}). For the distribution $g_{\beta}$
corresponding to the value $\beta=0.4,0.5, 0.6,0.8$ used in the tests,  we consider the following parameters
\begin{center}
\begin{tabular}{r|c|c|c|c|}
&$\alpha$&$\hat{\beta}$ &$\gamma$&$\delta$\\ \hline
$\beta=0.4$&$\beta$&1&$\gamma_c$&$\gamma_c-0.15$\\ \hline
$\beta=0.5$&$\beta$&1&$\gamma_c$&$\gamma_c$\\ \hline
$\beta=0.6$&$\beta$&1&$\gamma_c$&$\gamma_c$+0.15\\ \hline
$\beta=0.8$&$\beta$&1&$\gamma_c$&$\gamma_c$+0.5\\ \hline
\end{tabular}
\end{center}
having set $\gamma_c=\frac{1}{2}$. 
We  to briefly describe the algorithm:
\begin{algorithm}
\caption{Computation of $u_\b$}
\begin{algorithmic}[1]
\STATE{Define a uniform grid  of size $\Delta x$ in space   and $\Delta t$ in time;}
\STATE{Approximate $\int_0^\infty$ in \eqref{HJfnum} with $\int_0^M$ for a given parameter $M$;}
\STATE{Define  a partition of $[0,M]$ of size   $\Delta s = \frac{M}{N_{int}}$ given by the points   $s_k = k\Delta s$, $k = 0, \hdots, N_{int}$;}
\STATE{Compute the matrix $g_\b(s_k,t_j)$ by the toolbox   Stable Distribution;}
\STATE{Compute $\cE(s_k,t_j)$ by \eqref{pdf_trick};}
\STATE{Define $ U(i,k) := u (x_i,s_k)$ where $u$ is the solution of \eqref{HJ};}
\STATE{Compute the integral in \eqref{HJfnum} by midpoint rule;}
\RETURN $u_{\beta}(x_i,t_j).$
\end{algorithmic}
\end{algorithm}

\subsection{Test 1}
 Consider the Hamilton-Jacobi equation
\begin{equation}\label{HJ_ex1}
 \pd_t u+ |Du|^2=0, \qquad (x,t)\in\R\times (0,\infty),
\end{equation}
with the initial datum   $g(x) = -|x|^2$.  Then,  the solution of the problem is given  by
\begin{equation*} 
u(x,t) =  - (|x| + t)^2  \qquad (x,t)\in\R\times (0,\infty).
\end{equation*}
For $\b\in (0,1)$,   the solution of
\begin{equation}\label{HJf_ex1}
\fdz u+|Du|^2=0
\end{equation}
with the same  initial datum is given by
\begin{equation}\label{ex:solf}
u_\beta(x,t)  =-\int_0^\infty  (|x| + r)^2\cE_\b(r,t)dr.
\end{equation}
In order to highlight the impact of $\beta$ on the solution of the time-fractional Hamilton-Jacobi equation, we consider $x=0$ in \eqref{ex:solf} and study the evolution of $u(0,t)$ for $t\in(0,2]$ for different values of $\beta$. By \eqref{moment} and \eqref{ex:solf}, we have
$$u_\b(0,t)=\E[E_t^2]=C(2, \beta) t^{2\beta}.$$
Comparing the solutions of \eqref{HJ_ex1} and \eqref{HJf_ex1}, see figure \ref{test1}, we see that the effect of the Caputo derivative   is to induce a faster evolution for  small time, while a slower one as the time increases, a typical effect of the polynomial decay at infinity of the distribution of the subordinator.
\begin{figure}[!h]
\centering
\includegraphics[width=0.5\textwidth]{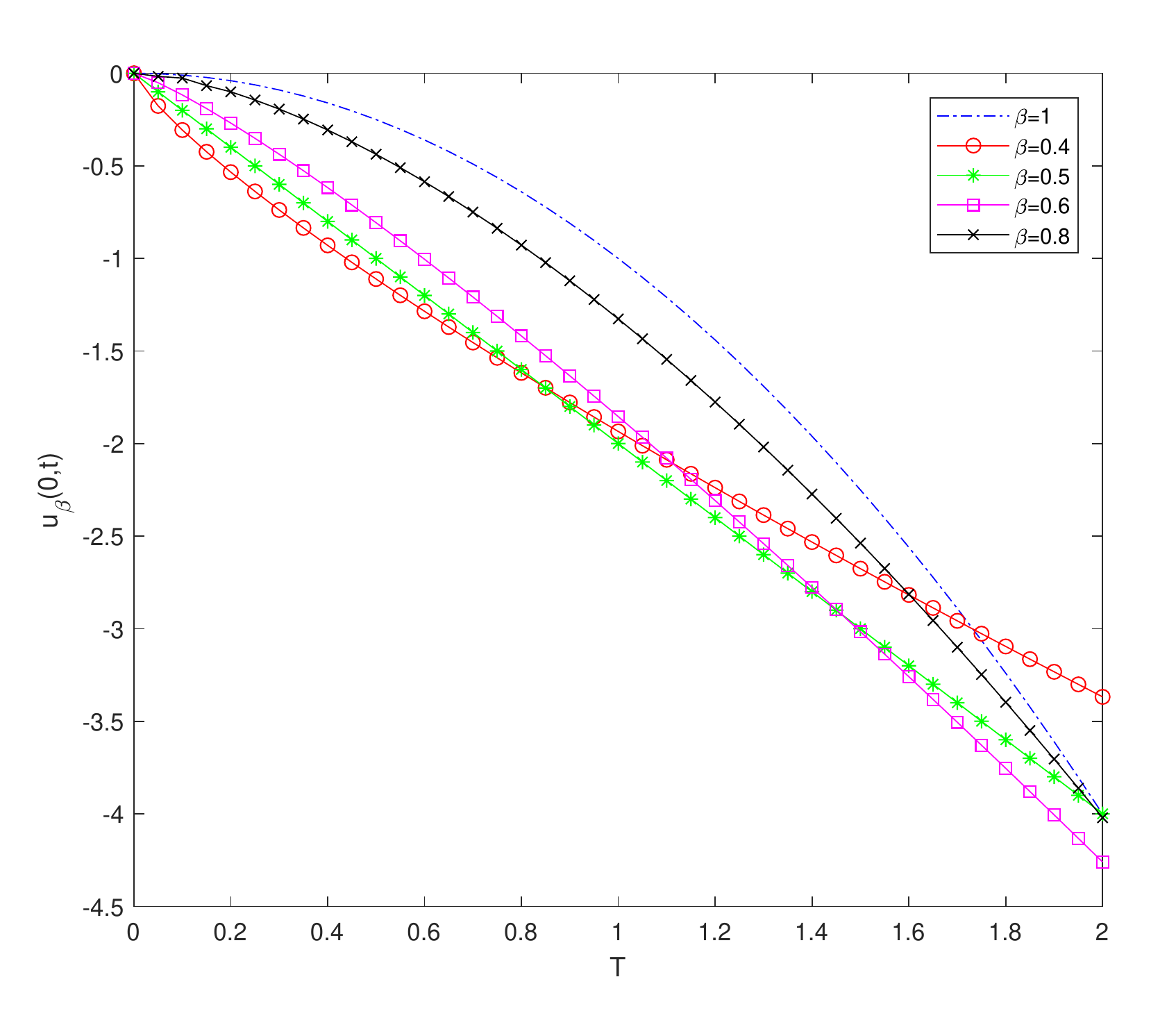}\,
\caption{$u_\b(0,t)$  for $t\in(0,2]$ and $\beta=0.4,\,0.5,\,0.6,\,0.8,\,1$.}
\label{test1}
\end{figure}

\subsection{Test 2}
We consider equation \eqref{HJ_ex1} with the initial condition
\begin{equation}
g(x) = \max\{0, x^2 - 1\}.
\end{equation}
In this case, the solution of \eqref{HJ_ex1}  is given by
\begin{equation}
u (x,t) = \max\left\{0, \frac{x^2}{1 + 2t} - 1\right\}.
\end{equation}
As before we compute  $u_\beta$ by means of formula  \eqref{HJfnum}.
\begin{figure}[h!]
\centering
\subfigure{\includegraphics[width=0.28\textwidth]{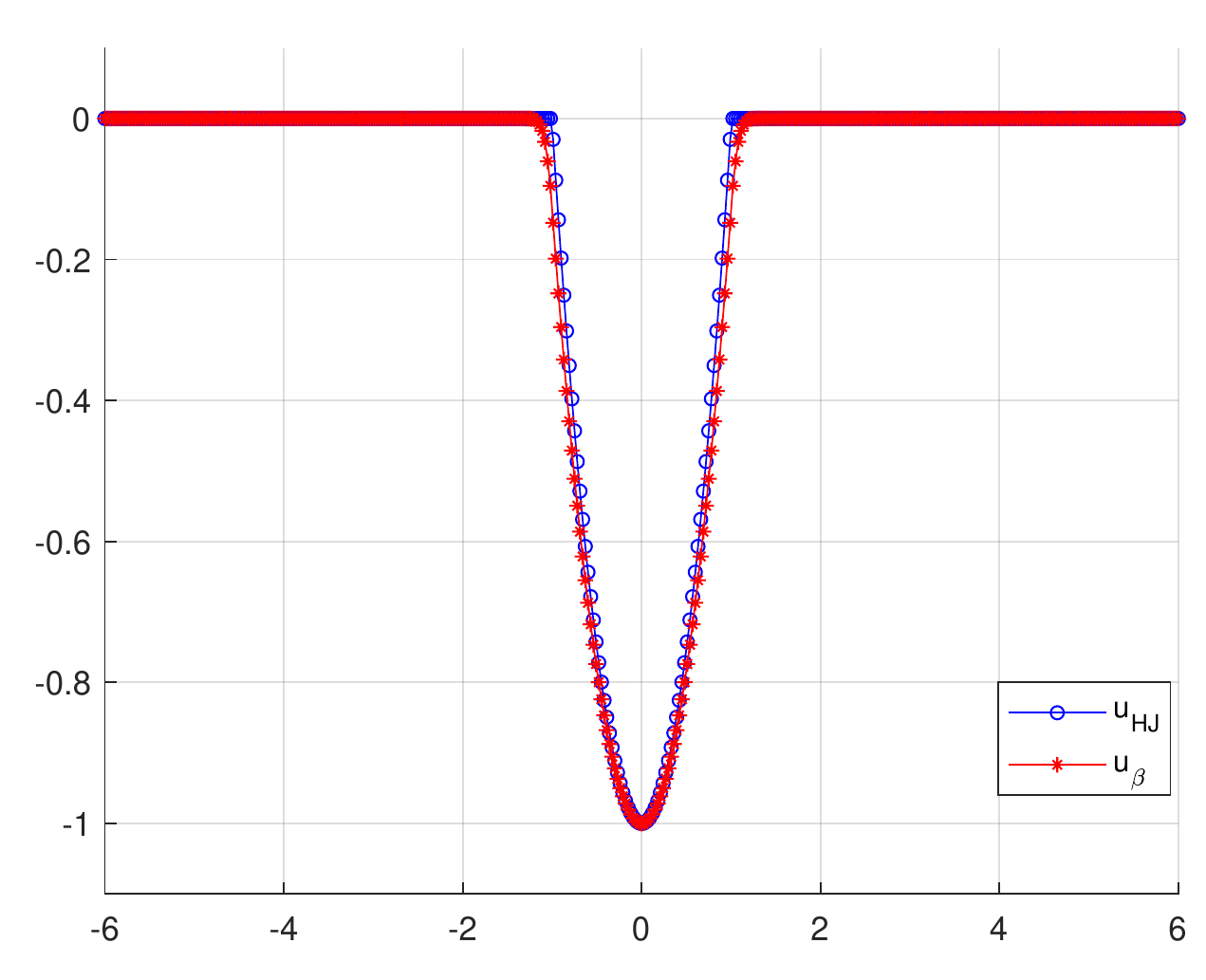}\,
\includegraphics[width=0.28\textwidth]{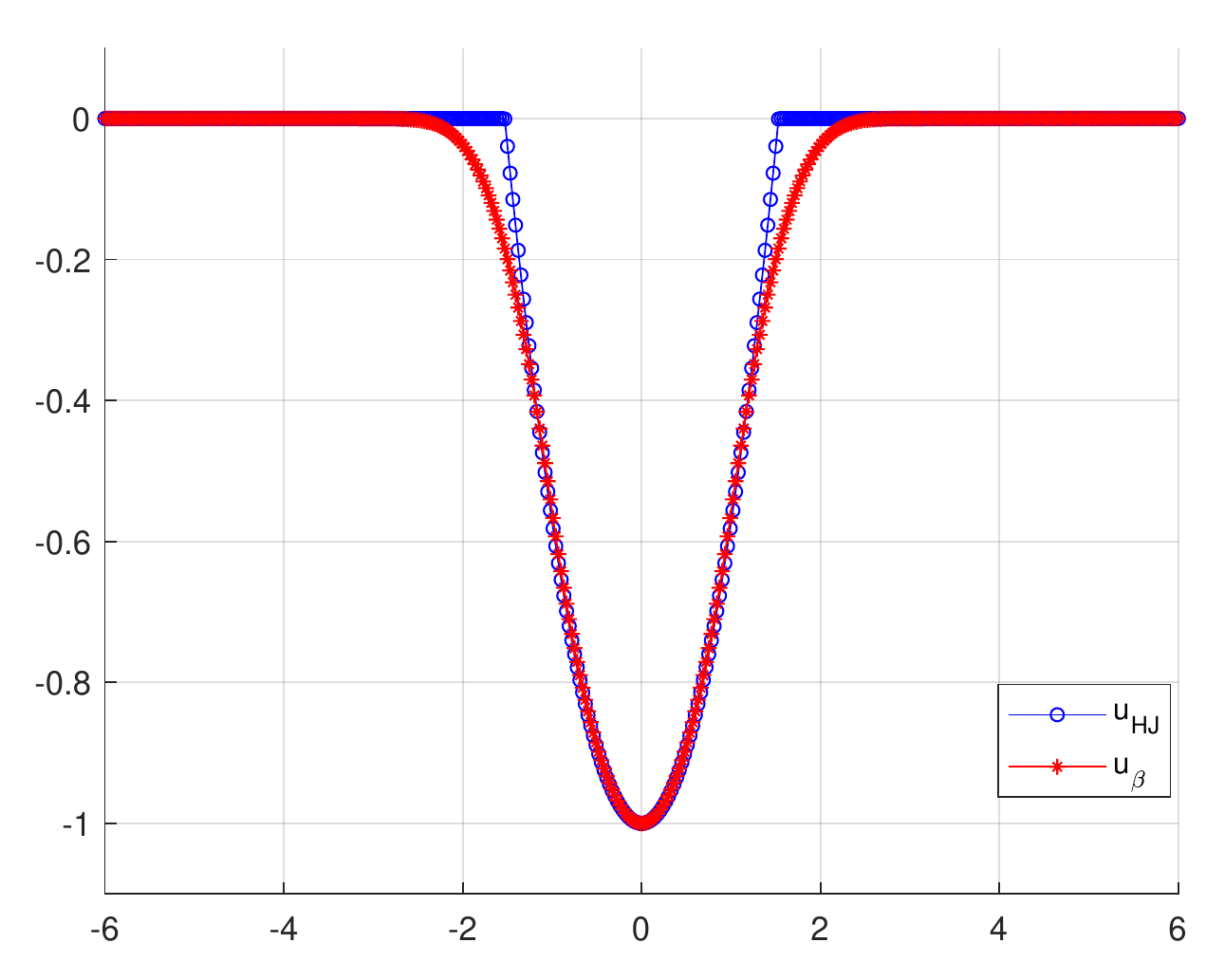}}
\subfigure{\includegraphics[width=0.28\textwidth]{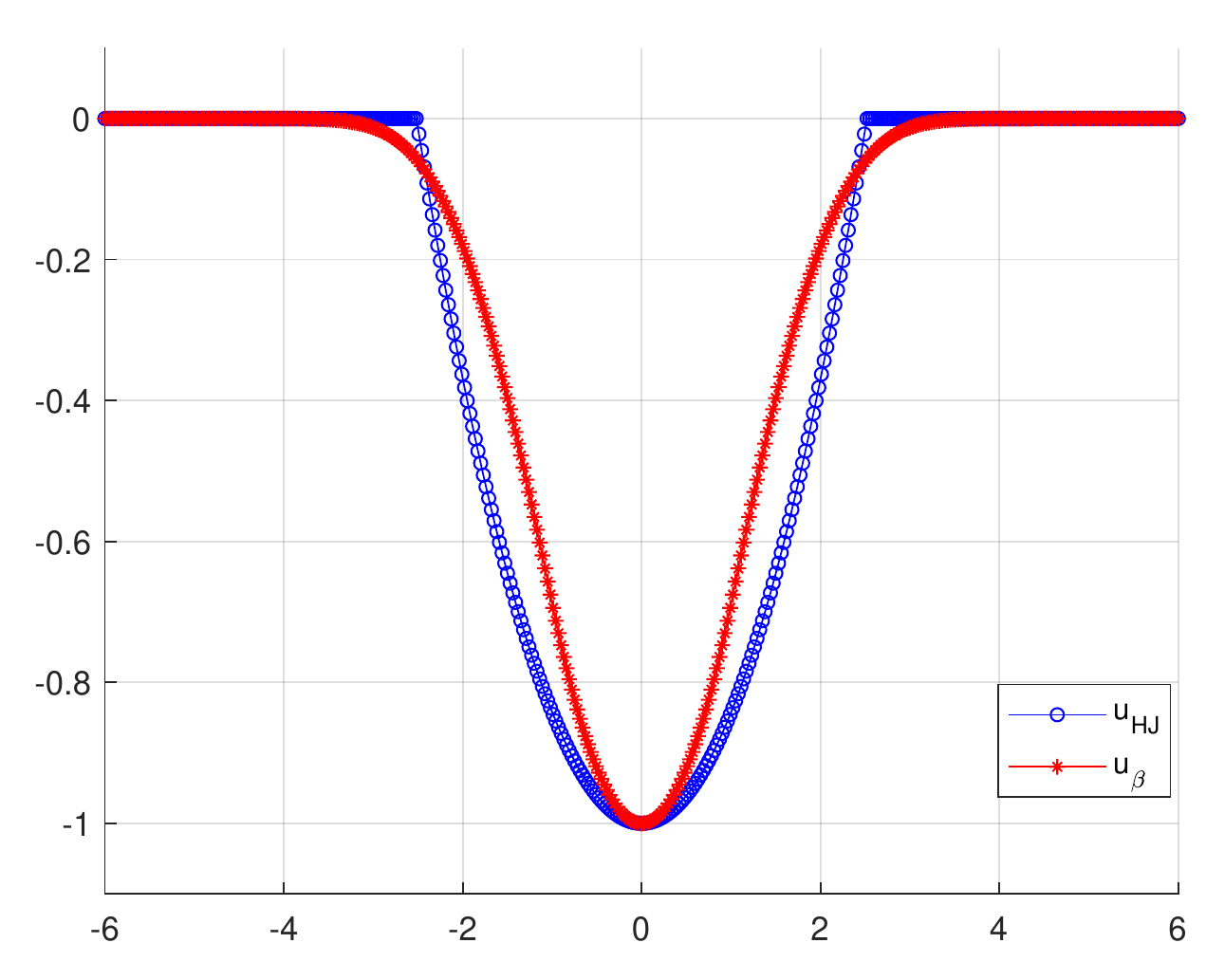}\,
\includegraphics[width=0.28\textwidth]{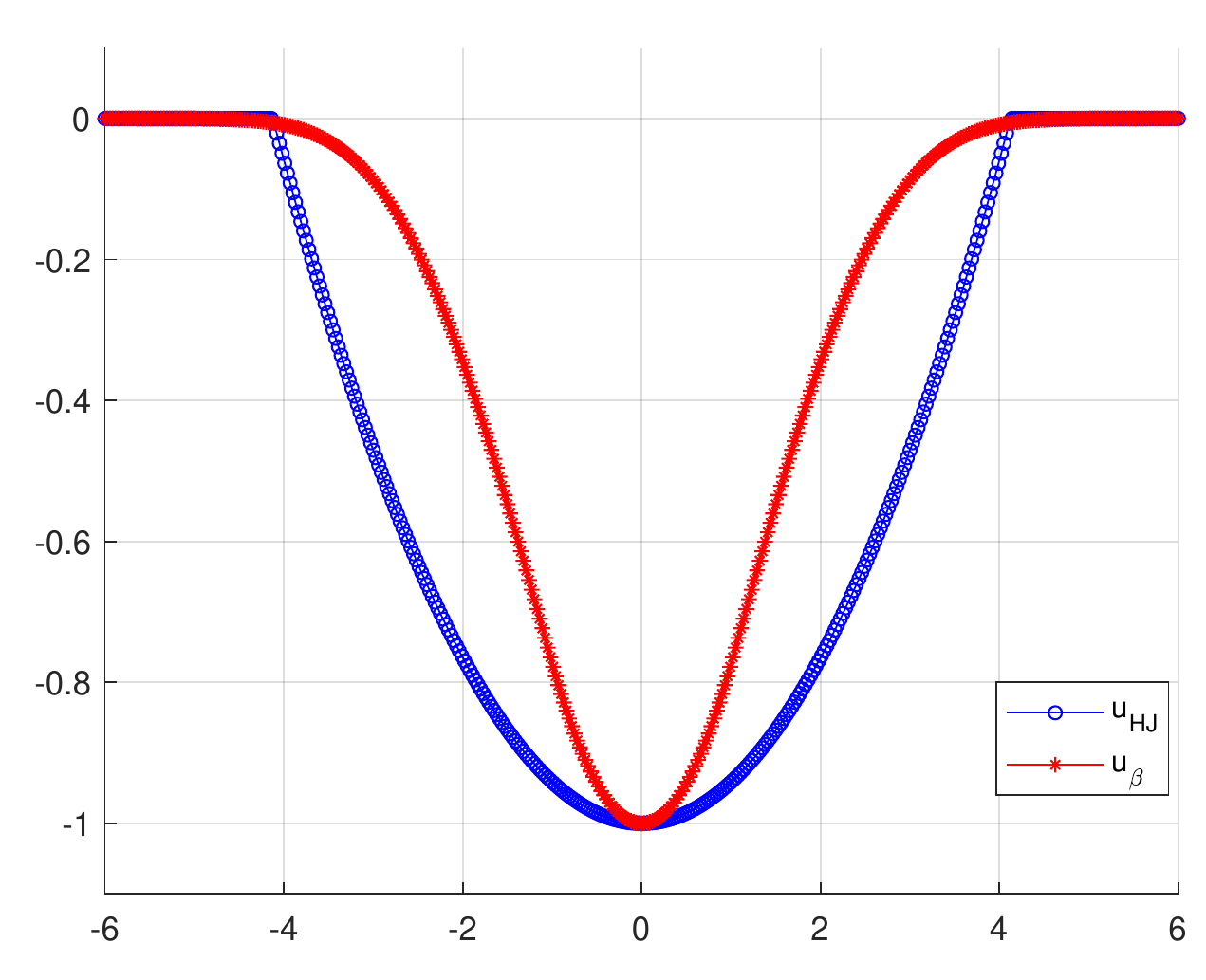}}
\caption{Comparison between $u$ and $u_{\beta}$ for $\beta=\frac 1 2$ at time $t=0.005$ (left-top), $t=0.5$ (right-top),
$t=2$ (left-bottom), and $t=8$ (right-bottom). }
\label{test2tpiccolo}
\end{figure}
Comparing the behavior of $u_\beta$ and $u$ in figure \ref{test2tpiccolo}, we can see that
 for small times the evolution of $u_{\beta}$ is faster than the one of  $u$, since  $u_{\beta}(x,t) \leq u (x,t)$ and $\textrm{supp}(u(t)) \subset\textrm{supp}(u_\beta(t))$ for $t \leq 0.5$. While the time increases, the evolution of $u_\b$ slows down with respect to the one of $u$.
It is also  interesting to observe the   more regular behavior of  $u_{\beta}$ in the space variable. Indeed  the initial edge of $g$ is instantaneously smoothed for the fractional equation, while
it persists for  \eqref{HJ_ex1} (see Figure \ref{test2uprimo}). We also observe that $u_\beta$ is not $C^2$ in space and  a ``memory'' of the initial edge of $g$
is preserved in the second derivative.
\begin{figure}[h!]
\centering
\includegraphics[width=0.4\textwidth]{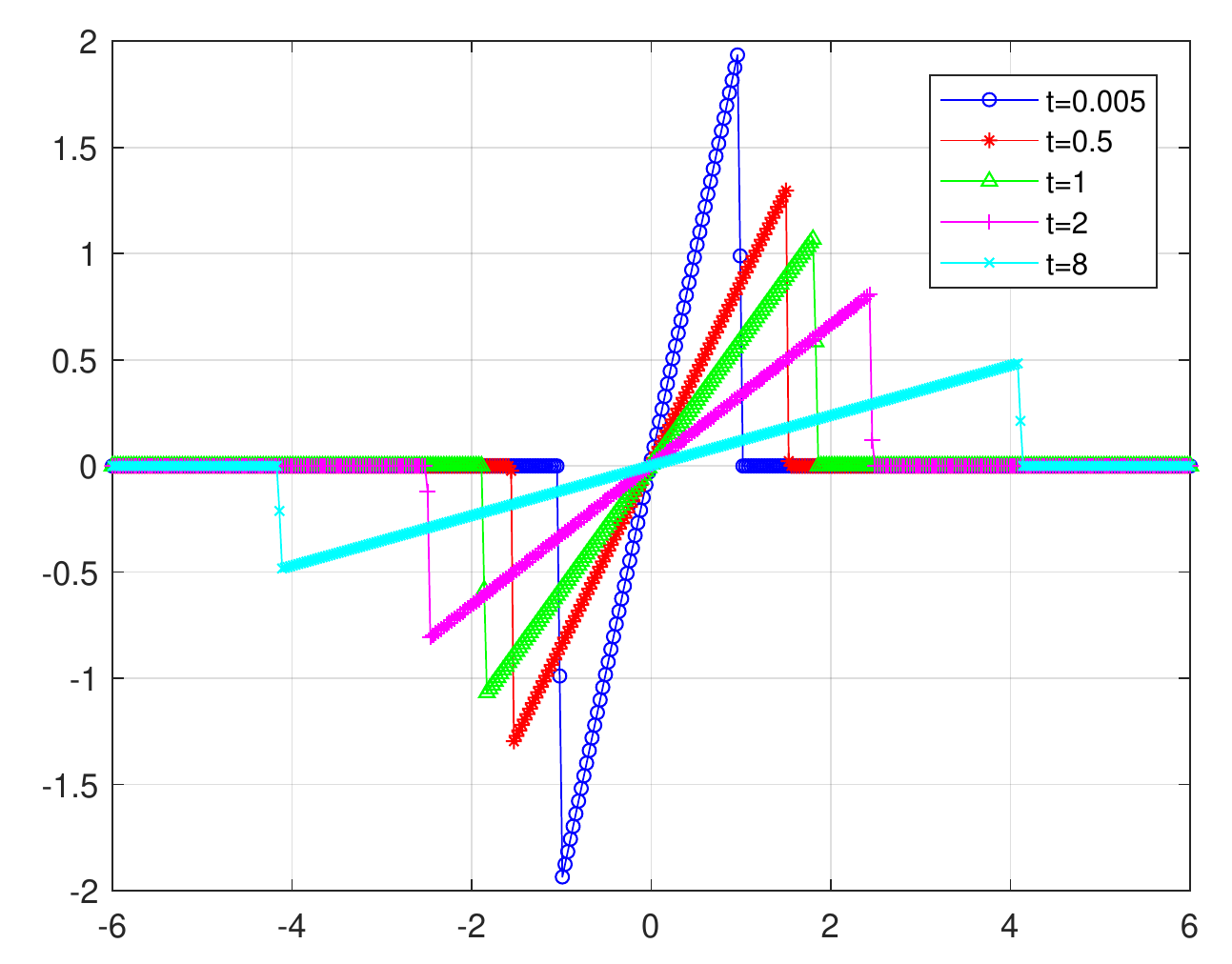}\,
\includegraphics[width=0.4\textwidth]{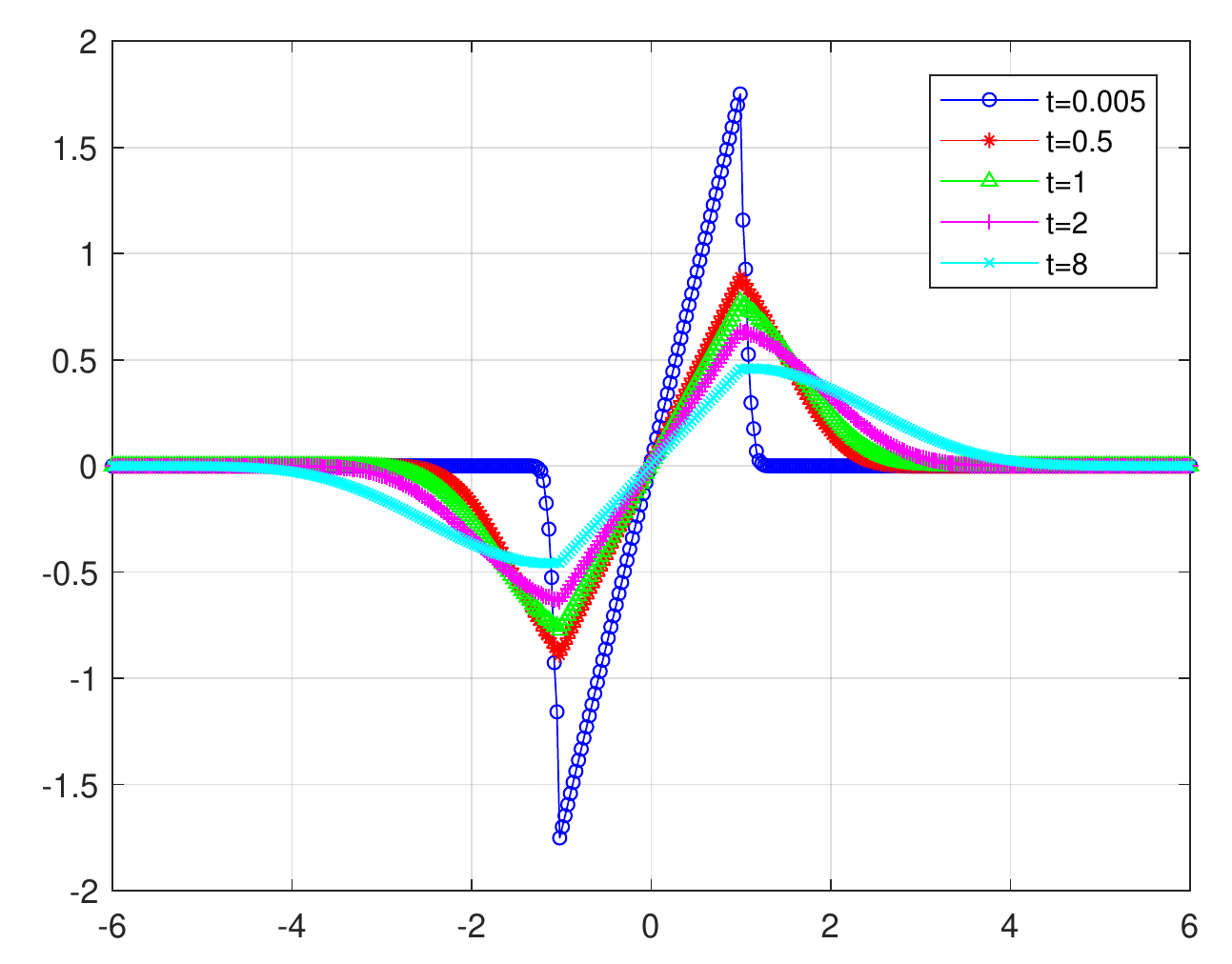}\,
\caption{Space derivative   of $u$ (left) and $u_{\beta}$ with $\beta=\frac 1 2$ (right), at different times.}
\label{test2uprimo}
\end{figure}


\subsection{Test 3}
The last test refers to the Hamilton-Jacobi equation
\begin{equation}\label{LS}
 \pd_t u+|Du|=0 \qquad (x,t)\in\R^2\times (0,\infty),
\end{equation}
which represents the  motion at a  constant speed of a level curve of the viscosity solution.
Even if the Hamiltonian $H(p)=|p|$ does not satisfy assumptions \eqref{hyp_H}, it is well known that formula \eqref{HL} is  still valid   and it simplifies in $u(x,t)=\min\{g(y):\, |x-y|\le t\}$. We consider the corresponding   time-fractional equation
\begin{equation}\label{LSf}
 \fdz u+|Du|=0 \qquad (x,t)\in\R^2\times (0,\infty),
\end{equation}
whose  solution if given by
$u_\b(x,t)=\E_{x,t}[\min\{g(y):\, |x-y|\le E_t\}]$.
In the first example, see Figure \ref{test3hjbfrac1}, we compare the evolution of a unitary circle for \eqref{LS} and  for \eqref{LSf} with $\b=\frac 1  2$.
Given  the initial datum  $g(x)=|x|^2-1$, we observe that
also in the fractional case  its evolution is given by circles of increasing radius, but the propagation speed is not uniform
and tends to slows down after some times. A similar property it is also observed   in the case of a initial front given by two circles, see Figure \ref{test3hjbfrac2}.
\begin{figure}[h!]
\centering
\includegraphics[width=0.45\textwidth]{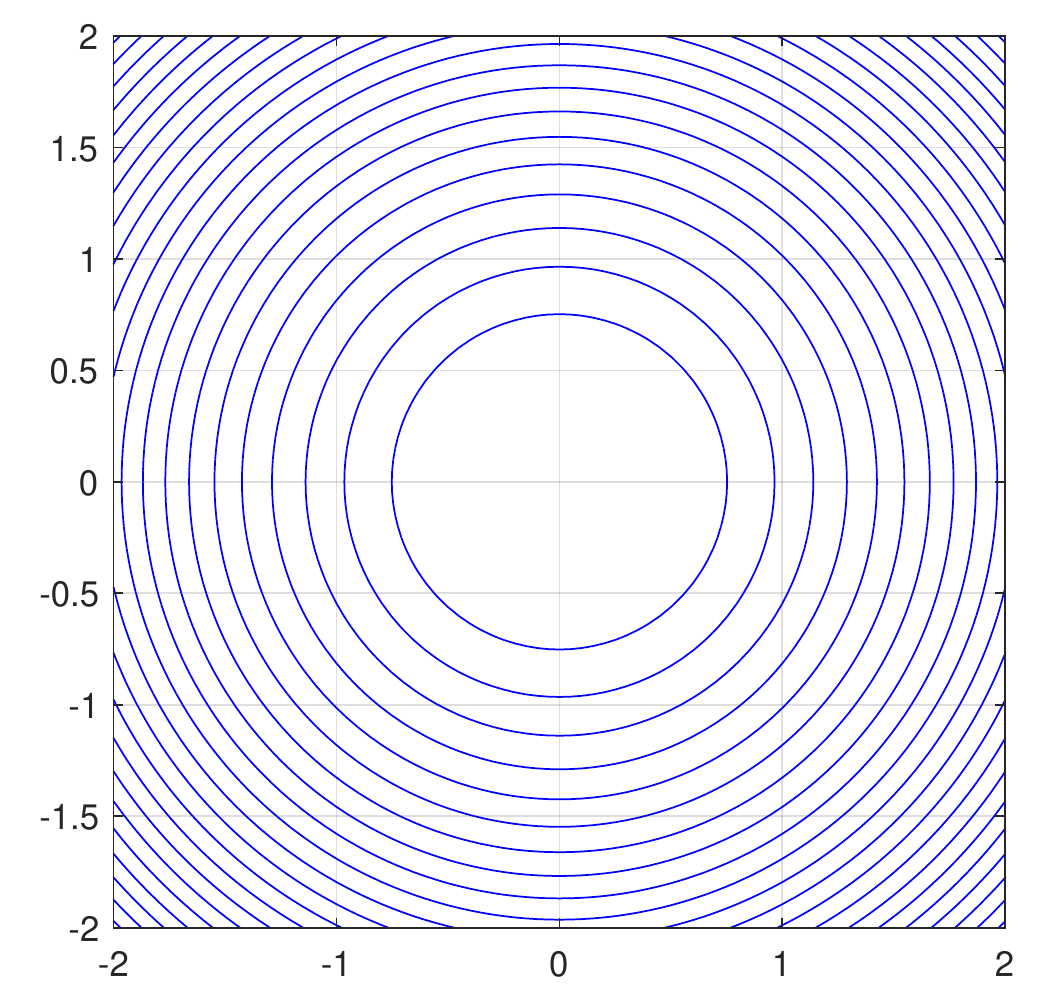}\,
\includegraphics[width=0.45\textwidth]{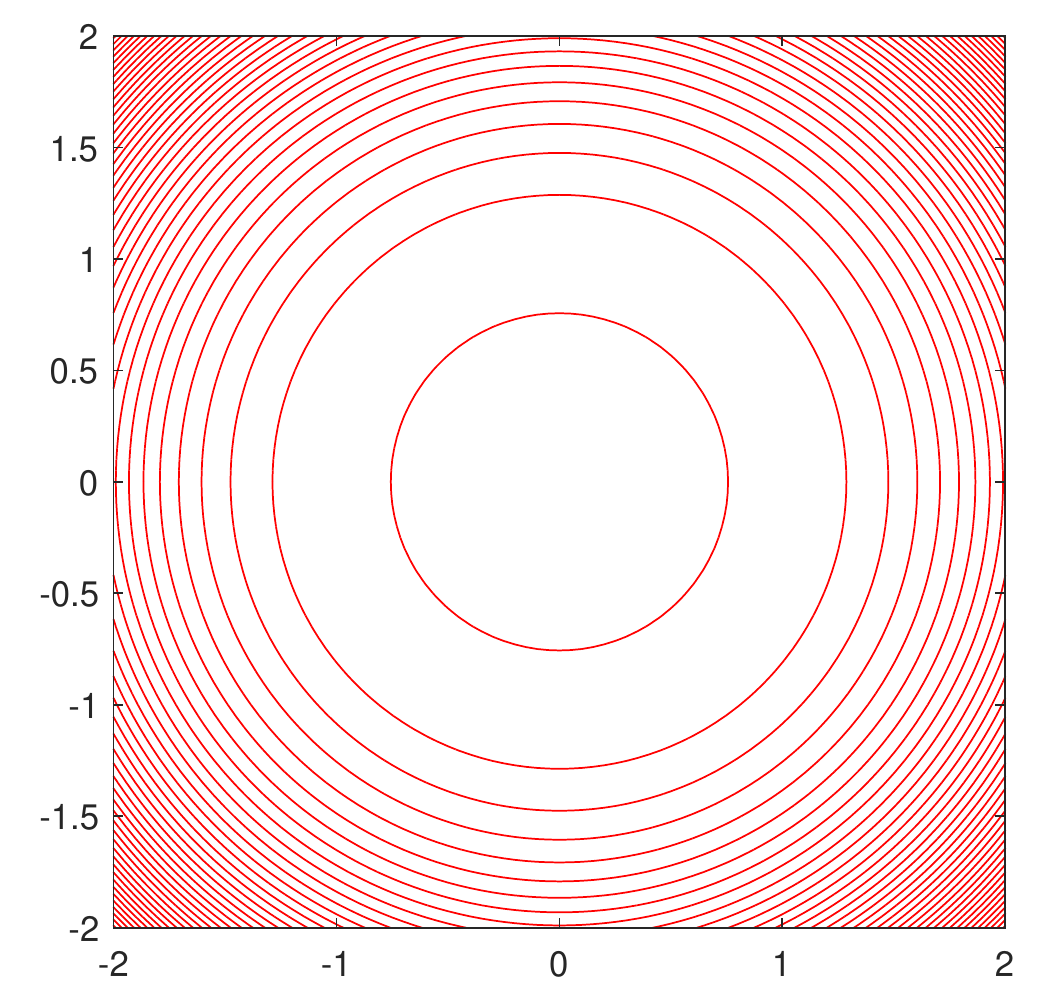}\,
\caption{Evolution of the 0-level sets of $u$ (left) and $u_{\beta}$ (right) with $\beta=\frac 1 2$, for $t\in[0,15]$.}
\label{test3hjbfrac1}
\end{figure}
\begin{figure}[!]
\centering
\includegraphics[width=0.45\textwidth]{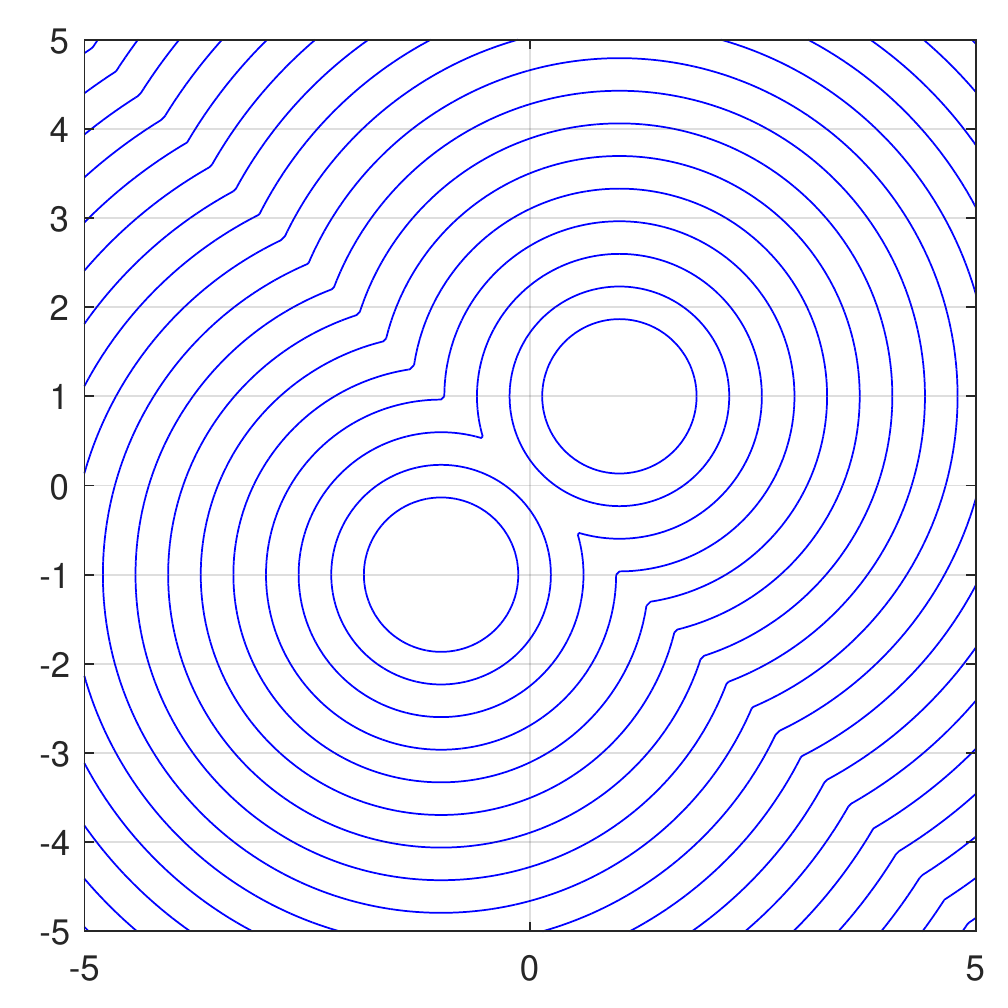}\,
\includegraphics[width=0.45\textwidth]{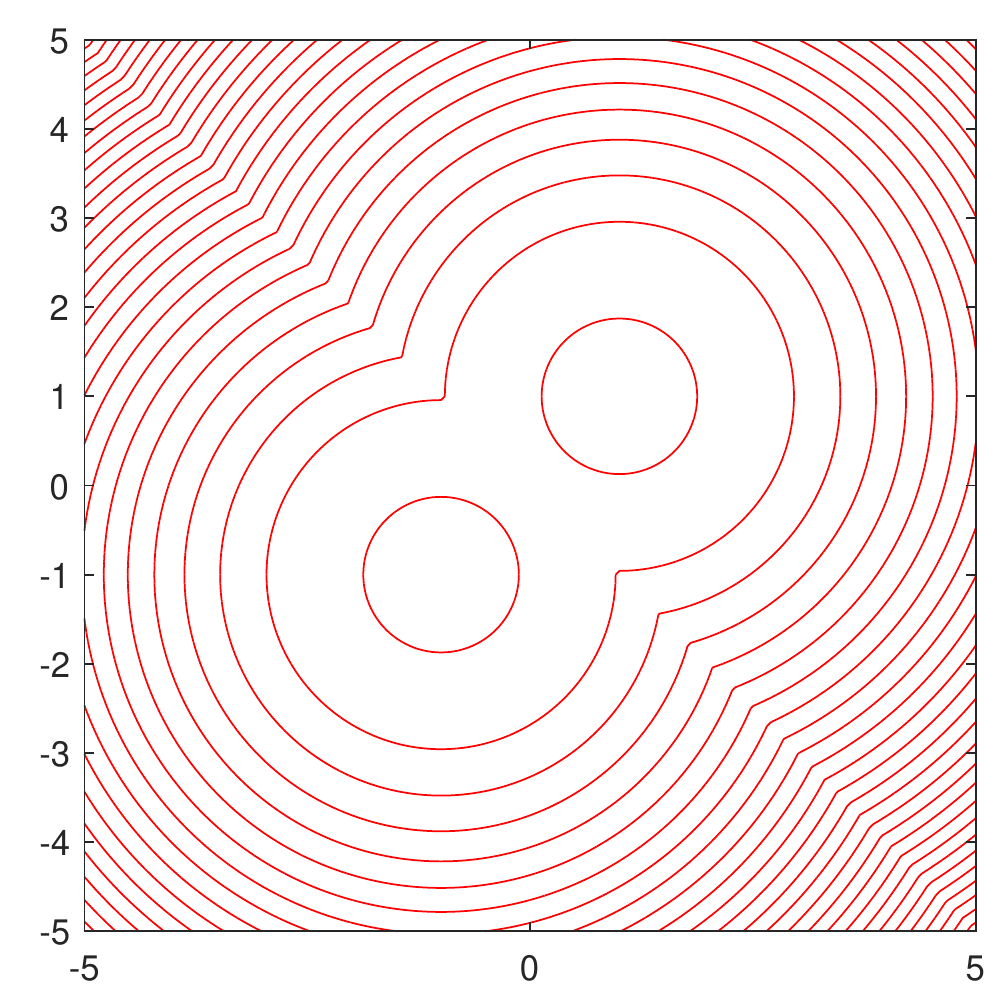}\,
\caption{Evolution of the 0-level sets of $u$ (left) and $u_{\beta}$ (right) with $\beta=\frac 1 2$  for $t\in [0,9]$.}
\label{test3hjbfrac2}
\end{figure}

\end{document}